\documentclass{amsart}[11pt]
\usepackage{amssymb}
\DeclareMathSymbol{\twoheadrightarrow} {\mathrel}{AMSa}{"10}

\def\Q{{\mathbf Q}}

\def\Z{{\mathbf Z}}
\def\C{{\mathbf C}}

\def\F{{\mathbf F}}

\def\H{{\mathrm H}}
\def\ST{{\mathbf S}}
\def\A{{\mathbf A}}
\def\Sn{{\mathbf S}_n}

\def\Gal{\mathrm{Gal}}

\def\Perm{\mathrm{Perm}}

\def\End{\mathrm{End}}

\def\Aut{\mathrm{Aut}}
\def\Hom{\mathrm{Hom}}

    \def\RR{\mathfrak{R}}

\def\fchar{\mathrm{char}}

\def\GL{\mathrm{GL}}

\def\Sp{\mathrm{Sp}}

        \def\K{\bar{K}}

\def\dim{\mathrm{dim}}

                           \def\Hdg{\mathrm{Hdg}}

\newtheorem{thm}{Theorem}[section]
\newtheorem{lem}[thm]{Lemma}
\newtheorem{cor}[thm]{Corollary}

\theoremstyle{definition}
\newtheorem{defn}[thm]{Definition}
\newtheorem{ex}[thm]{Example}
\newtheorem{exs}[thm]{Examples}

\newtheorem{rem}[thm]{Remark}

\title[Hodge classes on certain hyperelliptic prymians]{Hodge classes on certain hyperelliptic prymians}

\author{Yuri G. Zarhin}
\address{Department of Mathematics, Pennsylvania
State University, University Park, PA 16802, USA}
\email{zarhin\char`\@math.psu.edu}

\begin{document}
\maketitle

\section{Definitions and statements}

 Throughout this paper
$K$ is a field, $\K$ its algebraic closure and $\Gal(K)=\Aut(\K/K)$
the absolute Galois group of $K$.

 If $X$ is an abelian variety over
$\K$ then we write $\End(X)$ for the ring of all its
$\K$-endomorphisms
; the notation $1_X$ stands for the
identity automorphism of $X$. If $Y$ is an abelian variety over $\K$
then we write $\Hom(X,Y)$ for the corresponding group of all
$\K$-homomorphisms.

Let $f(x)\in K[x]$ be a polynomial of  degree $n\ge 2$ with
coefficients in $K$ and without multiple roots, $\RR_f\subset \K_a$
the ($n$-element) set of roots of $f$ and $K(\RR_f)\subset \K_a$ the
splitting field of $f$. We write $\Gal(f)=\Gal(f/K)$ for the Galois
group $\Gal(K(\RR_f)/K)$ and call it the Galois group of $f(x)$ over
$K$; it permutes roots of $f$ and may be viewed as a certain
permutation group of $\RR_f$, i.e., as  a subgroup of the group
$\Perm(\RR_f)\cong\Sn$ of permutation of $\RR_f$. (It is well known
that $\Gal(f)$ is transitive if and only if $f$ is irreducible.) Let
us put
$$g=\left[\frac{n-1}{2}\right].$$ Clearly, $g$ is a
nonnegative integer and either $n=2g+1$ or $n=2g+2$.

Let us assume that $\fchar(K)\ne 2$. We write $C_{f}$ for the genus
$g$ hyperelliptic $K$-curve $y^2=f(x)$ and $J(C_{f})$ for its
jacobian. Clearly, $J(C_{f})$ is a $g$-dimensional abelian variety
that is defined over $K$. In particular, $J(C_{f})=\{0\}$ if and
only if $n=2$. The abelian variety $J(C_{f})$ is an elliptic curve
if and only if $n=4$.

Let us assume that $K$ is a subfield of the field $\C$ of complex
numbers (and $\bar{K}$ is the algebraic closure of $K$ in $\C$).
Then one may view $J(C_{f})$ as a complex abelian variety and
consider its first rational homology group $\H_1(J(C_{f}),\Q)$ and
the Hodge group $\Hdg(J(C_{f}))$ of $J(C_{f})$, which is a certain
connected reductive algebraic $\Q$-subgroup of the general linear
group $\GL(\H_1(J(C_{f}),\Q))$
\cite{MumfordSh,Deligne,Ribet,ZarhinIzv,MZ2}. The canonical
principal polarization on $J(C_{f})$ gives rise to the nondegenerate
alternating bilinear form
$$\H_1(J(C_{f}),\Q) \times \H_1(J(C_{f}),\Q) \to \Q$$
and the corresponding symplectic group $\Sp( \H_1(J(C_{f}),\Q))$
contains $\Hdg(J(C_{f}))$ as a (closed) algebraic $\Q$-subgroup. In
addition, $\End(J(C_{f}))$ coincides with the endomorphism ring of
the complex abelian variety $J(C_{f})$ and $\End^{0}(J(C_{f}))$
coincides with the centralizer of $\Hdg(J(C_{f}))$ in
$\End_{\Q}(\H_1(J(C_{f}),\Q))$ \cite{MumfordSh,Ribet,MZ2}.

The following result was obtained by the author in \cite[Th.
2.1]{ZarhinMRL}, \cite[Sect. 10]{ZarhinMMJ}. (See also
\cite{ZarhinPLMS2}, \cite{ZarhinBSMF,ZarhinL}.)

\begin{thm}
\label{jacobian} Suppose that $K\subset \C$, $n \ge 5$ (i.e., $g \ge
2$) and $\Gal(f)=\ST_n$ or the alternating group $\A_n$. Then
$\End(J(C_{f}))=\Z$ and $\Hdg(J(C_{f}))=\Sp( \H_1(J(C_{f}),\Q))$.
Every Hodge class on each self-product of $J(C_{f})$ can be
presented as a linear combination of products of divisor classes. In
particular, the Hodge conjecture is valid for each self-product of
$J(C_{f})$.
\end{thm}

\begin{rem}
The assertion that $\Hdg(J(C_{f}))=\Sp( \H_1(J(C_{f}),\Q))$ was not
stated explicitly in \cite{ZarhinMMJ}. However, it follows
immediately from the description of the Lie algebra $\mathrm{mt}$ of
the corresponding Mumford-Tate group \cite[p. 429]{ZarhinMMJ} as the
direct sum of the scalars $\Q \mathrm{Id}$ and the Lie algebra of
the symplectic group, because the Lie algebra of the Hodge group
coincides with the intersection of Lie algebras of the Mumford-Tate
group and the symplectic group. (The same arguments prove the
equality $\Hdg(J(C_{f}))=\Sp( \H_1(J(C_{f}),\Q))$ for all $f(x)$
that satisfy the conditions of Theorem 10.1 of \cite{ZarhinMMJ}.)
\end{rem}

Our next result that was obtained in \cite[Th. 1.2 and Theorem
2.5]{ZarhinSh} deals with homomorphisms of hyperelliptic jacobians.

\begin{thm}
\label{homo} Suppose that $\fchar(K)\ne 2$, $n\ge 3$ and $m\ge 3$
are integers and let $f(x)$ and $h(x)$ be irreducible polynomials
over $K$ of degree $n$ and $m$ respectively. Suppose that
$$\Gal(f)=\ST_n, \ \Gal(h)=\ST_m$$
and the corresponding splitting fields $K(\RR_f)$ and $K(\RR_h)$ are
linearly disjoint over $K$. Then either
$$\Hom(J(C_f),J(C_h))=\{0\}, \ \Hom(J(C_h),J(C_f))=\{0\}$$
or $\fchar(K)>0$ and both $J(C_f)$ and $J(C_h)$ are supersingular
abelian varieties.
\end{thm}

\begin{rem}
\label{homoC} If $K\subset \C$ then  Theorem \ref{homo} implies that
(under its assumptions) there are no nonzero homomorphisms between
complex abelian varieties $J(C_f)$ and $J(C_h)$.
\end{rem}

The main results of the present  paper are the following statements.

\begin{thm}
\label{main}
 Suppose that  $n=2g+2=\deg(f)\ge 8$. Let
$\tilde{C}_f \to C_{f}$ be an unramified double cover of complex
smooth projective irreducible curves and let $P$ be the
corresponding Prym variety, which is a $(g-1)$-dimensional
(principally polarized) complex abelian variety.

If $\Gal(f)=\ST_n$
 then:

\begin{itemize}
\item
$\End(P)=\Z$ or $\Z\oplus \Z$.
\item
Every Hodge class on each self-product of $P$ can be presented as a
linear combination of products of divisor classes. In particular,
the Hodge conjecture holds true for each self-product of $P$.
\end{itemize}
\end{thm}

\begin{thm}
\label{main1}
 Suppose that  $n=2g+2\ge 10$, $K\subset\C$ and
$f(x)=(x-a)h(x)$ where $a\in K$ and $h(x)\in K[x]$ is an irreducible
degree $n-1$ polynomial with $\Gal(h)=\ST_{n-1}$.
  Let
$\tilde{C}_{f} \to C_{f}$ be an unramified double cover of complex
smooth projective irreducible curves and let $P$ be the
corresponding Prym variety, which is a $(g-1)$-dimensional
(principally polarized) complex abelian variety.

 Then:
\begin{itemize}
\item
$\End(P)=\Z$ or $\Z\oplus \Z$.
\item
  every Hodge class on each self-product of $P$ can be
presented as a linear combination of products of divisor classes. In
particular, the Hodge conjecture holds true for each self-product of
$P$.
\end{itemize}
\end{thm}

Our proof is based on the explicit description of Prym varieties of
hyperelliptic curves \cite{MumfordP,Dal} and our results about Hodge
groups of hyperelliptic jacobians mentioned above.

\begin{rem}
If $n=2g+2\le 10$ then $\dim(P)=g-1\le 3$. Notice that  if $A$ is a
complex abelian varietiy of dimension $\le 3$ then it is well known
that every Hodge class on each self-product of $A$ can be presented
as a linear combination of products of divisor classes \cite[Th.
0.1(iv)]{MZ2}. In particular, the Hodge conjecture holds true for
each self-product of $A$.
\end{rem}

The paper is organized as follows. In Section \ref{Galois} we
discuss an elementary construction from Galois theory  and apply it
in Section \ref{homohyper} to homomorphisms of hyperelliptic
jacobians. Section \ref{hyperhodge} deals with Hodge groups of
hyperelliptic jacobians. In Section \ref{hyperprym} we discuss
hyperelliptic prymians and prove the main results.

\section{Galois theory}
\label{Galois}

Throughout this Section, $K$ is an arbitrary field  and $n\ge 3$ is
an integer, $f(x)\in K[x]$ is a degree $n$ irreducible polynomial,
whose Galois group
$$\Gal(f)=\Gal(K(\RR_f)/K)$$
is the full symmetric group $\Perm(\RR_f)=\ST_n$. If $T\subset
\RR_f$ is a non-empty subset then we put
$$f_T(x)=\prod_{\alpha\in T}(x-\alpha)\in K(\RR_f)[x].$$
By definition,
$$\deg(f_T)=\#(T), \ \RR_{f_T}=T.$$
 We view $\Perm(T)$ as a subgroup of
$\Perm(\RR_f)=\Gal(f)$ that consists of all permutations that leave
invariant every element outside $T$.

\begin{rem}
\label{oneT} Let us consider the subfield $E_0=K(\RR_f)^{\Perm(T)}$
of $\Perm(T)$-invariants. Since $\Perm(T)$ leaves invariant
$T=\RR_{f_T}$,
$$f_T(x)\in E_0[x].$$
Clearly,
$$\Gal(K(\RR_f)/E_0)=\Perm(T)\subset \Perm(\RR_f)=\Gal(\RR_f/K).$$
 Let us prove that the splitting field
$E_0(\RR_{f_T})=E_0(T)$ of $f_T(x)$ over $E_0$ coincides with
$K(\RR_f)$. Indeed, $\Gal(K(\RR_f))/E_0(T))$ consists of all
elements of $\Perm(T)=\Gal(K(\RR_f)/E_0)$ that leave invariant every
element of $T$. Since every element of $\Perm(T)$ leaves invariant
every element of $\RR_f\setminus T$, $\Gal(K(\RR_f))/E_0(T))=\{1\}$,
i.e., $K(\RR_f)=E_0(T)=E_0(\RR_{f_T})$. This implies that the Galois
group
$$\Gal(E_0(\RR_{f_T})/E_0)=\Gal(K(\RR_f)/E_0)=\Perm(T).$$
\end{rem}

\begin{lem}
\label{key} Let $T$ and $S$ be two nonempty disjoint subsets of
$\RR_f$. Then there exists a field subextension $E/K \subset
K(\RR_f)/K$ that enjoys the following properties.

\begin{itemize}
\item[(i)]
The Galois group $\Gal(K(\RR_f)/E)$ coincides with the subgroup
$$\Gal(T)\times \Gal(S)\subset \Perm(\RR_f)=\Gal(K(\RR_f)/K),$$
which consists of all permutations that leave invariant $T,S$ and
every element outside $T \sqcup S$.
\item[(ii)]
Both $f_T(x)$ and $f_S(x)$ lie in $E[x]$, i.e., all their
coefficients belong to $E$.
\item[(iii)]
Let $E(\RR_T)=E(T)$ and $E(\RR_S)=E(S)$ be the splitting fields over
$E$ of $f_T(x)$ and $f_S(x)$ respectively. Then the natural
injective homomorphisms
$$\Gal(E(T)/E)\hookrightarrow \Perm(T), \ \Gal(E(S)/E)\hookrightarrow
\Perm(S)$$ are group isomorphisms, i.e.,
$$\Gal(E(T)/E)=\Perm(T), \ \Gal(E(S)/E)=
\Perm(S).$$
\item[(iv)]
$E(T)$ and $E(S)$  are linearly disjoint over $E$.
\item[(v)]
The compositum $E(T) E(S)$ coincides with $K(\RR_f)$.
\end{itemize}
\end{lem}

\begin{proof}
Recall that
$$\RR_{f_T}=T, \ \RR_{f_S}=S.$$
 We define $E$ as the subfield $K(\RR_f)^{\Perm(T)\times
\Perm(S)}$ of $\Perm(T)\times \Perm(S)$-invariants. Now Galois
theory gives us (i). The subgroup $\Perm(T)\times \Perm(S)$ leaves
invariant both sets $T=\RR_T$ and $S=\RR_S$. This implies that all
the coefficients of $f_T(x)$ and $f_S(x)$ are $[\Perm(T)\times
\Perm(S)]$-invariant, i.e., lie in $E$. This proves (ii). Clearly,
$$[K(\RR_f):E]=\#(\Perm(T)\times \Perm(S)).$$

Clearly, the subgroup of $\Perm(T)\times \Perm(S)$ that consists of
all permutations that act identically on $S$ coincides with
$\Perm(T)$. Similarly, the subgroup of $\Perm(T)\times \Perm(S)$
that consists of all permutations that act identically on $T$
coincides with $\Perm(S)$. This implies that
$$\Gal(E(T)/E)=[\Perm(T)\times \Perm(S)]/\Perm(S)=\Perm(T),$$
$$\Gal(E(S)/E)=[\Perm(T)\times \Perm(S)]/\Perm(T)=\Perm(S),$$
which proves (iii). This implies that
$$[E(T):E]=\#(\Perm(T)), \ [E(S):E]=\#(\Perm(S)).$$
Let $L$ be the compositum $E(T) E(S)$. Clearly, $L$ contains $T$,
$S$ and $E$. Therefore $\Gal(K(\RR_f)/L)$ consists of  elements of
$\Perm(T)\times \Perm(S)=\Gal(K(\RR_f)/E)$ that act identically on
$T$ and $T$. Since all elements of $\Perm(T)\times \Perm(S)$ act
identically on the complement to $T \sqcup S$, we conclude that
$\Gal(K(\RR_f)/L)=\{1\}$, i.e.,
$$K(\RR_f)=L=E(T) E(S).$$
This proves (v). We also obtain that
$$[E(T) E(S):E]=[K(\RR_f):E]=\#(\Perm(T)\times \Perm(S))=[E(T):E]
[E(S):E],$$ i.e.
$$[E(T) E(S):E]=[E(T):E]
[E(S):E],$$ which means that $E(T)/E$ and $E(S)/E$ are linearly
disjoint. This proves (iv).
\end{proof}

\section{Homomorphisms of hyperelliptic jacobians}
\label{homohyper} We keep the notation and assumptions of Section
\ref{Galois}. Also we assume that  $\fchar(K)\ne 2$.

\begin{thm}
\label{homoG}
  Let $T$ and $S$ be
disjoint nonempty subsets of $\RR_f$ and consider the hyperelliptic
curves
$$C_{f_T}:y^2=f_T(x), \ C_{f_S}:y^2=f_S(x)$$ and their jacobians
$J(C_{f_T})$ and $J(C_{f_S})$. Then either
$$\Hom(J(C_{f_T}), J(C_{f_S}))=\{0\}, \ \Hom(J(C_{f_S}),
J(C_{f_T}))=\{0\}$$ or $\fchar(K)>0$ and both $J(C_{f_T})$ and
$J(C_{f_S})$ are supersingular abelian varieties.
\end{thm}

\begin{proof}
If $\#(T)<3$ (resp.  $\#(S)<3$) then $C_{f_T}$ (resp. $C_{f_S}$) has
genus zero and $J(C_{f_T})=0$ (resp. $J(C_{f_S})=0$), which implies
that there are no nonzero homomorphisms between $J(C_{f_T})$ and
$J(C_{f_S})$. So, further we assume that
$$n_1:=\#(T)\ge 3, \ n_2:=\#(S)\ge 3.$$
By Lemma \ref{key}, there exists a field $E$ such that both $f_T(x)$
and $f_S(x)$ lie in $E[x]$, their Galois groups are
$\Perm(T)\cong\ST_{n_1}$ and $\Perm(S)\cong\ST_{n_2}$ respectively.
In addition, their splitting fields are linearly disjoint over $E$.
 Now the result follows from Theorem \ref{homo} applied to $E,
f_T(x), f_S(x)$ instead of $K, f(x), h(x)$.
\end{proof}

\begin{rem}
\label{odd} Suppose that  $m:=\#(S)=2r+1$ is odd and let $b$ be an
arbitrary element of $K$. Let us consider the hyperelliptic curve
$C^{b}_{f_S}:y^2=(x-b)f_S(x)$. By Remark \ref{oneT},  there exists a
field $E_0\subset K(\RR_f)$ such that  $f_S(x)$  lies in $E_0[x]$
and $\Gal(E_0(\RR_{f_S})/E_0)=\Perm(S)=\ST_m$. Then the standard
substitution \cite[p. 25]{ZarhinPLMS2}
$$x_1=\frac{1}{x-b}, \ y_1:=\frac{y}{(x-b)^{r+1}}$$
 gives us  a
degree $m$ irreducible polynomial $h(x_1)\in E[x_1]$ such that
$$E_0(\RR_h)=E_0(\RR_{f_S}), \
\Gal(E_0(\RR_h)/E)=\Gal(E_0(\RR_{f_S})/E)=\ST_m$$ and $C^{b}_{f_S}$
is $E_0$-birationally isomorphic to the hyperelliptic curve
$C_{h}:y_1^2=h(x_1)$. (It is assumed in \cite[p. 25]{ZarhinPLMS2}
that $m \ge 5$ but the substitution works for any positive odd $m$.)
\end{rem}

\begin{cor}
\label{homoGt}  Let $T$ and $S$ be disjoint nonempty subsets of
$\RR_f$ and assume that $\#(S)$ is odd. Let $b$ be an arbitrary
element of $K$. Let us consider the hyperelliptic curves
$$C_{f_T}:y^2=f_T(x), \ C^{b}_{f_S}:y^2=(x-b)f_S(x)$$ and their jacobians
$J(C_{f_T})$ and $J(C^{b}_{f_S})$. Then either
$$\Hom(J(C_{f_T}), J(C^{b}_{f_S}))=\{0\}, \ \Hom(J(C^{b}_{f_S}),
J(C_{f_T}))=\{0\}$$ or $\fchar(K)>0$ and both $J(C_{f_T})$ and
$J(C^{b}_{f_S})$ are supersingular abelian varieties.
\end{cor}

\begin{proof}
We may assume that both $m_1=\#(T)$ and $m_2=\#(S)$ are, at least,
$3$.
 Let $E$ be as in Lemma \ref{key}. In particular, the splitting fields
$E(T)$ and $E(S)$ are linearly disjoint over $E$ and
$$\Gal(E(T)/E)\cong \ST_{m_1}, \ \Gal(E(S)/E) \cong \ST_{m_2}.$$
 Using Remark \ref{homoGt} over
$E$ (instead of $E_0$), we obtain that there is a degree $m$
irreducible polynomial $h(x)\in E[x]$ such that
$$E(\RR_h)=E(\RR_{f_S}), \
\Gal(E(\RR_h)/E)=\Gal(E(\RR_{f_S})/E)=\ST_{m_2}$$ and $C^{b}_{f_S}$
is birationally $E$-isomorphic to the hyperelliptic curve
$C_{h}:y^2=h(x)$. Clearly, the jacobians $J(C^{b}_{f_S})$ and
$J(C_h)$ are isomorphic. Applying Theorem \ref{homo} to $E,
f_T(x),h(x)$, we conclude that either
$$\Hom(J(C_{f_T}), J(C_h))=\{0\}, \ \Hom(J(C_h)),
J(C_{f_T}))=\{0\}$$ or $\fchar(K)>0$ and both $J(C_{f_T})$ and
$J(C_{f_S})$ are supersingular abelian varieties. Since
$J(C^{b}_{f_S})$ and $J(C_h)$ are isomorphic, we are done.
\end{proof}

\begin{thm}
\label{homoGab} Let $T$ and $S$ be disjoint nonempty subsets of
$\RR_f$ and assume that both $\#(T)$ and $\#(S)$ are odd. Let $a$
and $b$ be arbitrary (not necessarily distinct) elements of $K$. Let
us consider the hyperelliptic curves
$$C^{a}_{f_T}:y^2=(x-a)f_T(x), \ C^{b}_{f_S}:y^2=(x-b)f_S(x)$$ and their jacobians
$J(C^{a}_{f_T})$ and $J(C^{b}_{f_S})$. Then either
$$\Hom(J(C^{a}_{f_T}), J(C^{b}_{f_S}))=\{0\}, \ \Hom(J(C^{b}_{f_S}),
J(C^{a}_{f_T}))=\{0\}$$ or $\fchar(K)>0$ and both $J(C^{b}_{f_T})$
and $J(C^{a}_{f_S})$ are supersingular abelian varieties.
\end{thm}

\begin{proof}
We may assume that both $m_1:=\#(T)$ and $m_2:=\#(S)$ are, at least,
$3$. Again,  let $E$ be as in Lemma \ref{key}. Applying Remark
\ref{homoGt} two times over $E$ (instead of $E_0$) to the
polynomials $(x-a)f_T(x)$ and $(x-b)f_S(x)$, we conclude that
 there are degree $m$ irreducible polynomials $h_1(x)\in E[x]$ and
$h_2 (x)\in E[x]$ such that
$$E(\RR_{h_1})=E(T), \ E(\RR_{h_2})=E(S),$$
$$\Gal(E(\RR_{h_1})/E)=\ST_{m_1}, \ \Gal(E(\RR_{h_2})/E)=\ST_{m_2},$$
$C^{a}_{f_T}$ is $E$-birationally isomorphic to $C_{h_1}$ and
$C^{b}_{f_T}$ is $E$-birationally isomorphic to $C_{h_2}$. Clearly,
$J(C^{a}_{f_T})\cong J(C_{h_1})$ and $J(C^{b}_{f_S})\cong
J(C_{h_2})$. Applying Theorem \ref{homo} to $E, h_1(x),h_2(x)$, we
conclude that either
$$\Hom(J(C_{h_1}), J(C_{h_2}))=\{0\}, \ \Hom(J(C_{h_2})),
J(C_{h_1}))=\{0\}$$ or $\fchar(K)>0$ and both $J(C_{h_1})$ and
$J(C_{h_2})$ are supersingular abelian varieties. The rest is clear.

\end{proof}

\begin{rem}
Let $K_2/K$ be the only quadratic subextension of $K(\RR_f)/K$.
Clearly, $K_2(\RR_f)=K(\RR_f)$ and the Galois group
$\Gal(K_2(\RR_f)/K_2)$ coincides with the alternating group $\A_n$.
\end{rem}

\begin{thm}
\label{ellipticZ} Suppose that $\fchar(K)\ne 2$. Let $T \subset
\RR_f$ be a $4$-element subset. Let us consider
 the corresponding elliptic curve
$$C_{f_T}:y^2=f_T(x)$$
and its jacobian $J(C_{f_T})$. If $n \ge 8$ then one of the
following conditions holds:

\begin{itemize}
\item
 $\End(J(C_{f_T}))=\Z$ for all $T$.
\item
$\fchar(K)>0$ and all $J(C_{f_T})$'s are supersingular elliptic
curves mutually isomorphic over $\K$.
\end{itemize}
\end{thm}

\begin{proof}
Let $j_T$ be the $j$-invariant of the elliptic curve $J(C_{f_T})$
(\cite{Tate}, \cite[Ch. III, Sect. 2]{Knapp}). Clearly,
$$j_T \in K(T)\subset K(\RR_f)$$ and
$$j_{\sigma T}=\sigma j_T \ \forall \sigma\in
\Gal(K(\RR_f)/K)=\Gal(f).$$ Suppose that $J(C_{f_T})$ admits complex
multiplication. Then one of the following two conditions holds.

\begin{itemize}
\item[(i)]
$p=\fchar(K)>0$. Then a classical result of M. Deuring asserts that
$j_T$ is {\sl algebraic}, i.e., lies in a finite field $\F_q$ where
$q$ is a power of the prime $p$. (See \cite{Deuring}, \cite[Sect.
3.2]{Oort}, \cite[Ch. 13, Sect. 5]{Lang}.) In particular, $K(j_T)/K$
is an abelian field extension.

\item[(ii)]
$\fchar(K)=0$. Then there exists an imaginary quadratic field $k$
such that $\End^0(J(C_{f_T}))=k$. In addition, a classical result of
the theory of complex multiplication asserts that $j_T$ is an
algebraic number such that the field extension $k(j_T)/k$ is
abelian. (See \cite[Sect. 5.4]{Shimura}, \cite[Ch. 10, Sect.
3]{Lang}.)
\end{itemize}

Let us consider the overfield $K^{\prime}$ of $K$ that is defined as
follows. If $\fchar(K)>0$ then $K^{\prime}=K_2$. If $\fchar(K)=0$
then $K^{\prime}$ is the compositum $K_2 k$ of $K_2$ and the
imaginary quadratic field $k$; in particular, $K^{\prime}$ contains
$k$.

Since $\A_n=\Gal(K^{\prime}(\RR_f)/K^{\prime})$ is simple
nonabelian, the field extension $K^{\prime}(\RR_f)/K^{\prime}$ does
not contain nontrivial abelian subextensions.
However, $j_T \in K^{\prime}(\RR_f)$ and the field (sub)extension
$K^{\prime}(j_T)/K^{\prime}$ is abelian. This implies that
 this
subextension is trivial, i.e., $j_T\in K^{\prime}$. This means that
for all $\sigma \in \Gal(K^{\prime}(\RR_f)/K^{\prime})=\A_n$
$$j_T=\sigma j_T=j_{\sigma T}.$$
Since $n \ge 8$, the permutatation group $\A_n$ is $4$-transitive
and  therefore the jacobians $J(C_{f_T})$'s are mutually isomorphic
over $\K$ for all $4$-element subsets $T\subset \RR_f$.

Let $T_1$ and $T_2$ be two {\sl disjoint} $4$-element subsets of
$\RR_f$. (Since $n\ge 8$, such $T_1$ and $T_2$ do exist.) Applying
Theorem \ref{homoG} to $T_1$ and $T_2$ (instead of $T$ and $S$) and
taking into account that $J(C_{f_{T_1}})$ and $J(C_{f_{T_2}})$ are
isomorphic over $\K$ (i.e.,
$\Hom(J(C_{f_{T_1}}),J(C_{f_{T_2}}))\ne\{0\}$), we conclude that
$\fchar(K)>0$ and both $J(C_{f_{T_1}})$ and $J(C_{f_{T_2}}))$ are
supersingular elliptic curves.
\end{proof}

\begin{thm}
\label{ellipticZ1} Suppose that $\fchar(K)\ne 2$. Let $a$ be an
arbitrary element of $K$. Let $T \subset \RR_f$ be a $3$-element
subset. Let us consider
 the corresponding elliptic curve
$$C^{a}_{f_T}:y^2=(x-a)f_T(x)$$
and its jacobian $J(C^{a}_{f_T})$. If $n \ge 6$ then one of the
following conditions holds:

\begin{itemize}
\item
 $\End(J(C^{a}_{f_T}))=\Z$ for all $T$.
\item
$\fchar(K)>0$ and all $J(C^{a}_{f_T})$'s are supersingular elliptic
curves mutually isomorphic over $\K$.
\end{itemize}
\end{thm}

\begin{proof}
 Let $j_{T,a}$ be the $j$-invariant of the elliptic curve
$J(C^{a}_{f_T})$.
Clearly,
$$j_{T,a} \in K(T)\subset K(\RR_f)$$ and
$$j_{\sigma T,a}=\sigma j_{T,a} \ \forall \sigma\in
\Gal(K(\RR_f)/K)=\Gal(f).$$ Suppose that $J(C^{a}_{f_T})$ admits
complex multiplication. Then, as in the proof of Theorem
\ref{ellipticZ}, there exists an overfield $K^{\prime}\supset K_2$
such that either $K^{\prime}=K_2$ or $K^{\prime}$ is a quadratic
extension of $K_2$ and in both cases $K^{\prime}(j_{T,a})\subset
K^{\prime}(\RR_f)$ and the field (sub)extension
$K^{\prime}(j_{T,a})/K^{\prime}$ is abelian. Again,
$\A_n=\Gal(K^{\prime}(\RR_f)/K^{\prime})$ is simple nonabelian and
therefore there are no nontrivial abelian subextensions of
$K^{\prime}(\RR_f)/K^{\prime}$. This implies
  that $j_{T,a}\in K^{\prime}$, i.e.,
 for all $\sigma \in
\Gal(K^{\prime}(\RR_f)/K^{\prime})=\A_n$
$$j_{T,a}=\sigma j_{T,a}=j_{\sigma T, a}.$$

Since $n \ge 6$, the permutatation group $\A_n$ is $3$-transitive
and therefore the jacobians $J(C^{a}_{f_T})$'s are mutually
isomorphic over $\K$ for all $3$-element subsets $T\subset \RR_f$.
Let $T_1$ and $T_2$ be two {\sl disjoint} $3$-element subsets of
$\RR_f$. (Since $n\ge 6$, such $T_1$ and $T_2$ do exist.) Applying
Theorem \ref{homoGab} to $T_1,a$ and $T_2,a$ (instead of $T,a$ and
$S,b$) and taking into account that $J(C^{a}_{f_{T_1}})$ and
$J(C^{a}_{f_{T_2}})$ are isomorphic over $\K$ (i.e.,
$\Hom(J(C_{f_{T_1}}),J(C_{f_{T_2}}))\ne\{0\}$), we conclude that
$\fchar(K)>0$ and both $J(C^{a}_{f_{T_1}})$ and
$J(C^{a}_{f_{T_2}}))$ are supersingular elliptic curves. The rest is
clear.
\end{proof}

\section{Hodge groups of hyperelliptic jacobians}
\label{hyperhodge} We keep the notation and assumptions of Sections
\ref{Galois} and \ref{homohyper}. Also we assume that $K\subset \C$.

\begin{defn}
We say (as in \cite[Sect. 1.8]{MZ2}) that a complex abelian variety
$X$ satisfies property (D) if every Hodge class on each self-product
$X^r$ of $X$ can be presented as a linear combination of products of
divisor classes. If this condition  is satisfied then the Hodge
conjecture is true for all  $X^r$.
\end{defn}

\begin{rem}
Abelian varieties that satisfy (D) are also called {\sl stably
nondegenerate} \cite{hazamaT}; see also \cite{murty}.
\end{rem}

\begin{ex}
\label{DE} If $Y$ is an elliptic curve over $\C$ with $\End(Y)=\Z$
then it is well  known \cite[Th. 0.1(iv)]{MZ2} that $Y$ satisfies
(D) and $\Hdg(Y)=\Sp(\H_1(Y,\Q))$.
\end{ex}

\begin{thm}
\label{nonsimple} Let $X_1$ and $X_2$ be complex abelian varieties
of positive dimension and $X=X_1\times X_2$. Suppose that
$$\End(X_1)=\Z, \ \End(X_2)=\Z, \ \Hom(X_1,X_2)=\{0\}.$$

Then:

\begin{enumerate}
\item
$\End(X_1\times X_2)=\Z\oplus \Z$.
\item
If both $X_1$ and $X_2$ satisfy (D) then $\Hdg(X)= \ \Hdg(X_1)\times
\Hdg(X_2)$ and $X$ satisfies (D).
\end{enumerate}
\end{thm}

\begin{proof}
(i) is obvious. (ii) follows from  \cite[Th. 0.1 and Prop.
1.8]{Hazama} (see also Theorem 3.2(i) of \cite{MZ2}).
\end{proof}

\begin{rem}
Since $\End(X_i)=\Z$ and $X_i$ satisfies (D),
$$\Hdg(X_i)=\Sp(H_1(X_i,\Q)$$
\cite{hazamaT,murty}. (See also \cite[Sect. 1.8]{MZ2}.)
\end{rem}

\begin{lem}
\label{jacprym} Suppose that
$T$ is a subset of $\RR_f$ with $\#(T)\ge 5$. Let us consider the
hyperelliptic curve $C_{f_T}:y^2=f_T(x)$ and its jacobian
$J(C_{f_T})$. Then $\End(J(C_{f_T}))=\Z$ and $\Hdg(J(C_{f_T}))=\Sp(
\H_1(J(C_{f_T}),\Q))$. In addition, $J(C_{f_T})$ satisfies (D).

\end{lem}

\begin{proof}
Let us put $m=\#(T)$. We have  $m \ge 5$ and $\deg(f_T)=m\ge 5$.

 By Remark \ref{oneT}, there exists a (sub)field
$$E_0\subset K(\RR_f)\subset\K\subset \C$$
such that $f_T(x)\in E_0(T)$ and the Galois group of $f_T(x)$ over
$E_0$ is $\Perm(T)\cong \ST_m$. Now the result follows from Theorem
\ref{jacobian} applied to $m,E_0, f_T(x)$ instead of $n,K,f(x)$.
\end{proof}

\begin{lem}
\label{jacfam} Suppose that
$T$ is a subset of $\RR_f$ with $\#(T)\ge 5$.  Suppose that $m$ is
odd and let $a$ be an arbitrary element of $K$. Let us consider the
hyperelliptic curve $C^{a}_{f_T}:y^2=(x-a)f_T(x)$ and its jacobian
$J(C^{a}_{f_T})$. Then $\End(J(C^{a}_{f_T}))=\Z$ and
$\Hdg(J(C^{a}_{f_T}))=\Sp( \H_1(J(C^{a}_{f_T}),\Q))$. In addition,
 $J(C^{a}_{f_T})$ satisfies (D).

\end{lem}

\begin{proof}
By Remark \ref{odd}, there exists a field $$E_0\subset
K(\RR_f)\subset\K\subset\C$$ and a degree $m$ irreducible polynomial
$h(x)\in E_0[x]$ such that
$$E(\RR_h)=E(\RR_{f_T}), \
\Gal(E(\RR_h)/E)=\Gal(E(\RR_{f_T})/E)=\ST_m$$ and $C^{a}_{f_S}$ is
birationally $E$-isomorphic to to the hyperelliptic curve
$C_{h}:y^2=h(x)$. Clearly, the jacobians $J(C^{a}_{f_S})$ and
$J(C_h)$ are isomorphic. It follows from Lemma \ref{jacprym} applied
to $m,E_0,h(x)$ (instead of $n,K,f(x)$) that $\End(J(C_{h}))=\Z$,
$\Hdg(J(C_h)=\Sp( \H_1(J(C_h),\Q))$ and $J(C_{h})$ satisfies (D).
Since $J(C^{a}_{f_S})$ and $J(C_h)$ are isomorphic, we
are done.
\end{proof}

\section{Prym varieties}
\label{hyperprym}

 Following \cite{MumfordP,Dal}, let us give an
explicit description of hyperelliptic prymians $P$, assuming that
$\fchar(K)\ne 2$. Suppose that $n=2g+2\ge 6$  and $$f(x)\in
K[x]\subset \K[x]$$ is a degree $n$ polynomial without multiple
roots. Let us split the $n$-element set $\RR_f$ of roots of $f(x)$
into a {\sl disjoint} union
$$\RR_f =\RR_1 \sqcup \RR_2$$
of {\sl non-empty} sets $\RR_1$ and $\RR_2$ of {\sl even}
cardinalities $n_1$ and $n_2$ respectively.  Further we assume that
$$n_1 \ge n_2\ge 2.$$
 and put
$$f_1(x)=\prod_{\alpha\in\RR_1}(x-\alpha), \
f_2(x)=\prod_{\alpha\in\RR_2}(x-\alpha).$$ We have $n_1+n_2=n$ and
define nonnegative integers $g_1$ and $g_2$ by
$$n_1=2g_1+2, \ n_2=2g_2+2.$$
Clearly,
$$g_1+g_2=g-1.$$
Let us consider the hyperelliptic curves $C_{f_1}:y^2=f_1(x)$ and
$C_{f_2}:y^2=f_2(x)$ of genus $g_1$ and $g_2$ respectively and
corresponding hyperellptic jacobians $J(C_{f_1})$ and $J(C_{f_2})$
of dimension $g_1$ and $g_2$ respectively. Then the prymians $P$ of
$C_f: y^2=f(x)$ are just the products $J(C_{f_1})\times J(C_{f_2})$
for all the partitions $\RR_f =\RR_1 \sqcup \RR_2$.

Now Theorem \ref{main} becomes an immediate corollary of the
following statement.

\begin{thm}
\label{help} Suppose that  $n=2g+2\ge 8$, $K\subset\C$ and
$\Gal(f)=\ST_n$. Let us put $P=J(C_{f_1})\times J(C_{f_2})$. Then:
\begin{itemize}
\item[(i)]
$\Hom(J(C_{f_1}),J(C_{f_2}))=\{0\}, \
\Hom(J(C_{f_2}),J(C_{f_1}))=\{0\}$.
\item[(ii)]
Suppose that $g_i \ge 1$, i.e., $n_i \ge 4$. Then
$$\End(J(C_{f_i}))=\Z, \ \Hdg(J(C_{f_i}))=\Sp(
\H_1(J(C_{f_i}),\Q))$$ and   $J(C_{f_i})$ satisfies (D).
\item[(iii)]
If $n_2=2$ then $P=J(C_{f_1})$. In particular, $\End(P)=\Z$,
$\Hdg(P)=\Sp( \H_1(P,\Q))$ and $P$ satisfies (D).
\item[(iv)]
If $n_2 \ge 4$ then $\End(P)=\Z\oplus\Z$,
$$\Hdg(P)=\Hdg(J(C_{f_1}))\times \Hdg(J(C_{f_2}))=\Sp(
\H_1(J(C_{f_1}),\Q))\times \Sp( \H_1(J(C_{f_2}),\Q))$$ and $P$
satisfies (D).
\end{itemize}
\end{thm}

\begin{proof}[Proof of Theorem \ref{help}]
The assertion (i) follows from Theorem \ref{homoG}.

If $n_i\ge 6$ then (ii) follows from Lemma \ref{jacprym}. Suppose
that $n_i=4$. Then $J(C_{f_i})$ is an elliptic curve. It follows
from Theorem \ref{ellipticZ} that $\End(J(C_{f_i}))=\Z$. Now the
assertion about its Hodge group and property (D)  follows from
Example \ref{DE}. This completes the proof of (ii).

Let us prove (iii). If  $n_2=2$  then $J(C_{f_2})=0$ and therefore
$P=J(C_{f_1})$. Now the assertion follows from (ii).

Let us prove (iv). We assume that
$$n_1 \ge n_2 \ge 4.$$
By already proven (i) and (ii),
 $$\End(J(C_{f_1}))=\Z, \ \End(J(C_{f_2}))=\Z,  \
 \Hom(J(C_{f_1}),J(C_{f_2}))=\{0\}.$$
Now (iv) follows from Theorem \ref{nonsimple} applied to
$X_1=J(C_{f_1})$, $X_2=J(C_{f_2})$ and $X=P$.
\end{proof}

Theorem \ref{main1} is an immediate corollary of the following
statement.

\begin{thm}
\label{help1} Suppose that  $n=2g+2\ge 10$, $K\subset\C$ and
$f(x)=(x-a)h(x)$ where $a\in K$ and $h(x)\in K[x]$ is an irreducible
degree $(n-1)$ polynomial with $\Gal(h)=\ST_{n-1}$.  Let us put
$P=J(C_{f_1})\times J(C_{f_2})$. Then:
\begin{itemize}
\item[(i)]
$\Hom(J(C_{f_1}),J(C_{f_2}))=\{0\}, \
\Hom(J(C_{f_2}),J(C_{f_1}))=\{0\}$.
\item[(ii)]
Suppose that $g_i \ge 1$, i.e., $n_i \ge 4$. Then
$$\End(J(C_{f_i}))=\Z, \ \Hdg(J(C_{f_i}))=\Sp(
\H_1(J(C_{f_i}),\Q))$$ and $J(C_{f_i})$ satisfies (D).

\item[(iii)]
If $n_2=2$ then $P=J(C_{f_1})$. In particular, $\End(P)=\Z$,
$\Hdg(P)=\Sp( \H_1(P,\Q))$ and $P$ satisfies (D).

\item[(iv)]
If $n_2 \ge 4$ then $\End(P)=\Z\oplus\Z$,
$$\Hdg(P)=\Hdg(J(C_{f_1}))\times \Hdg(J(C_{f_2}))=\Sp(
\H_1(J(C_{f_1}),\Q))\times \Sp( \H_1(J(C_{f_2}),\Q))$$
 and $P$ satisfies (D).
\end{itemize}
\end{thm}

\begin{proof}
Clearly, $\RR_f=\RR_h \sqcup\{a\}$; in particular, $a$ belongs to
precisely one of  $\RR_1$ and $\RR_2$.  Suppose that $a$ lies in
$\RR_j$ and does {\sl not} belong to $\RR_k$ and put
$$T=\RR_k\subset\RR_h, \ S=\RR_j\setminus \{a\}\subset\RR_h.$$ Now the assertion (i) follows from
Corollary \ref{homoGt} applied to  $h(x)$ (instead of $f(x)$).

If $n_k\ge 6$ then the assertion (ii) for $J(C_{f_k})$ follows from
Lemma \ref{jacprym}. Suppose that $n_k=4$, i.e.,  $J(C_{f_k})$ is an
elliptic curve. Then it follows from Theorem \ref{ellipticZ} applied
to $m=n-1\ge 9$ and $h(x)$ (instead of $f(x)$) that
$\End(J(C_{f_k}))=\Z$. Now the assertion about its Hodge group and
property (D) follows from Example \ref{DE}.

 If $n_j\ge 6$ then the assertion (ii) for $J(C_{f_j})$ follows from Lemma \ref{jacfam}.
Suppose that $n_j=4$, i.e.,  $J(C_{f_j})$ is an elliptic curve. Then
it follows from Theorem \ref{ellipticZ1} applied to $m=n-1$ and
$h(x)$ (instead of $f(x)$) that $\End(J(C_{f_j}))=\Z$. Now the
assertion about its Hodge group and property (D) follows from
Example \ref{DE}. This ends the proof of (ii).

The proof of the remaining assertions (iii) and (iv) goes literally
as the proof of the corresponding assertions of Theorem \ref{help}.
\end{proof}

\begin{exs}
Let us take $K=\Q$ and $f_n(x)=x^n-x-1$. It is known \cite[p.
42]{SerreGalois} that $\Gal(f_n)=\ST_n$. Let $a$ be a rational
number. Suppose that $n=2g+2$ and let us consider the hyperelliptic
genus $g$ curves $C_{f_n}: y^2=f_n(x)$ and $C^{a}_{f_{n-1}}:
y^2=(x-a) f_{n-1}(x)$. Then:

\begin{itemize}
\item[(i)]
If $n=2g+2\ge 8$ then all $(2^{2g}-1)$ Prym varieties $P$ of
$C_{f_{n}}$ satisfy (D). Among them there are exactly $n(n-1)/2$
complex abelian varieties with $\End(P)=\Z$; for all others
$\End(P)=\Z\oplus \Z$.
\item[(ii)]
If $n=2g+2\ge 10$ then all $(2^{2g}-1)$ Prym varieties $P$ of
$C^{a}_{f_{n-1}}$ satisfy (D). Among them there are exactly
$n(n-1)/2$ complex abelian varieties with $\End(P)=\Z$; for all
others $\End(P)=\Z\oplus \Z$.
\end{itemize}
\end{exs}

\begin{ex}
Let $z_1, \dots , z_n$ be algebraically independent (transcendental)
complex numbers and $L=\Q(z_1, \dots , z_n)\subset \C$ the
corresponding subfield of $\C$, which is isomorphic to the field of
rational functions in $n$ variables over $\Q$. Let $K\subset L$ be
the (sub)field of symmetric rational functions. Then
 $$f(x)=\prod_{i=1}^n (x-z_i)\in K[x], \ \RR_f=\{z_1, \dots , z_n\}, \
 \Gal(f)=\ST_n.$$
Suppose that $n=2g+2$ and let us consider the hyperelliptic genus
$g$ curve $C_f:y^2=f(x)$.

If $g\ge 3$ (i.e., $n\ge 8$) then all $(2^{2g}-1)$ Prym varieties
$P$ of $C_{f}$ satisfy (D). Among them there are exactly $n(n-1)/2$
complex abelian varieties with $\End(P)=\Z$; for all others
$\End(P)=\Z\oplus \Z$.

If $n=6$ (i.e., $g=2$) then all fifteen Prym varieties $P$ are
elliptic curves $y^2=\prod_{z\in T}(x-z)$ where $T$ is a $4$-element
subset of $\{z_1, \dots , z_6\}$. The algebraic independence of
$z_1, \dots , z_6$ implies that the $j$-invariants of these elliptic
curves are transcendental numbers and therefore all $P$ have no
complex multiplication, i.e., $\End(P)=\Z$.

\end{ex}

\begin{rem}
The property (D) and equality $\End(P)=\Z$  for {\sl general} (not
necessarily unramified) Prym varieties $P$ of arbitrary smooth
projective curves were proven in \cite{Biswas}.
\end{rem}

\end{document}